\documentclass{amsart}
\usepackage{amsthm,amsfonts,amsmath,amscd,amssymb,latexsym,epsfig}
\usepackage[title,titletoc]{appendix}

\newcommand{\cx}{{\mathbb C}}

\newcommand{\rank}{\operatorname{rank}}

\newcommand{\ol}{\overline}

\numberwithin{equation}{section}

\newtheorem{theorem}{Theorem}[section]

\newtheorem{theorema}{Theorem}

\newtheorem{lemma}[theorem]{Lemma}

\newtheorem{proposition}[theorem]{Proposition}

\theoremstyle{remark}

\newtheorem{remark}[theorem]{Remark}

\newcommand{\oC}{{\mathbb{C}}}

\newcommand{\oP}{{\mathbb{P}}}
\newcommand{\oQ}{{\mathbb{Q}}}
\newcommand{\oR}{{\mathbb{R}}}

\newcommand{\oZ}{{\mathbb{Z}}}

\newcommand{\sO}{{\mathcal{O}}}

\begin{document}

\title{Nonnegative polynomials from vector bundles on real curves}
\author{Roger Bielawski}
\address{Institut f\"ur Differentialgeometrie\\
Universit\"at Hannover\\ Welfengarten 1\\ D-30167 Hannover}


\begin{abstract} We observe that the $E$-resultant of a very ample rank $2$ vector bundle $E$ on a real projective curve (with no real points) is nonnegative when restricted to the space of real sections. Moreover, we show that if $E$ has a section vanishing at exactly two points and
the degree $d$ of $E$ satisfies $d(d-6)\geq 4(g-1)$, then this polynomial cannot be written as a sum of squares.
\end{abstract}

\maketitle

\thispagestyle{empty}

\section{Introduction}

Ever since Hilbert's seminal 1888 paper \cite{Hil} it has been known that there exist real polynomials in $n$ variables ($n\geq 3$), which are nonnegative on all of $\oR^n$, but cannot be written as a sum of squares of polynomials. Hilbert's proof was nonconstructive and the first explicit example was found only in 1967. The subject of nonnegative polynomials has become a rapidly evolving area of real algebraic geometry, with links and applications to
semidefinite programming and  polynomial optimization problems: see, e.g., the  monograph \cite{Mar}.
\par
In this note, we are going to show that a class of nonnegative polynomials, which cannot be written as sums of squares, corresponds to very ample rank $2$ vector bundles on real projective curves with no real points (the vector bundles must satisfy some additional constraints; see below). The simplest vector bundle satisfying the necessary conditions is $\sO(3)\oplus \sO(3)$ on $\cx \oP^1$ equipped with the antipodal map, and we compute explicitly the corresponding nonnegative polynomial which is not a sum of squares  (Example 3.3). It is of degree $6$ in $8$ variables, with $224$ terms. 
\par
Let us explain the correspondence between vector bundles and polynomials in greater detail.
Let $E$ be a globally generated nontrivial vector bundle of rank $2$ on a smooth connected projective curve $C$ (over $\oC$). The cone 
$$ V(E)=\{s\in H^0(C,E);\; \exists_{z\in C}\; s(z)=0\}
$$
of sections of $E$ which vanish somewhere is an irreducible hypersurface of $H^0(C,E)$, and hence the zero set of an irreducible homogeneous polynomial $R_E$.  Following Gelfand, Kapranov, and Zelevinsky \cite[Ch. 3]{GKZ}, we call $R_E$ the {\em $E$-resultant}. 
\par
Suppose now, in addition, that $C$ is a real curve, i.e. $C$ is equipped with an antiholomorphic involution $\sigma$, and that $E$ is a $\sigma$-bundle, i.e. there is an antilinear bundle isomorphism $\tau:E\to E$ which covers $\sigma$ and satisfies $\tau^2=1$.  We can then restrict $R_E$ to the real vector space $H^0(C,E)^\tau$ of $\tau$-invariant sections and obtain (after multiplying by a constant) a real polynomial $\rho_E$ on $H^0(C,E)^\tau$.
\par 
The real polynomial $\rho_E$ is the subject of this note. It is of course determined only up to a nonzero multiple, and we can assume that $\rho_E$ takes at least one positive value.
Our main result is:

\begin{theorema} Let $C$ be a  smooth connected projective curve of genus $g$, equipped with an antiholomorphic involution $\sigma$ with no fixed points. Let $E$ be a rank $2$ and degree $d$ very ample $\sigma$-bundle on $C$. Then:
\begin{itemize}
 \item[(i)] $\rho_E$ is a nonnegative homogeneous polynomial of degree $d$;
\item[(ii)] $\rho_E$ vanishes on a set of codimension $2$, unless  $E$ is $\sO(1)\oplus \sO(1)$ on $\oP^1$. At a generic point of $\rho_E^{-1}(0)$, the rank of the Hessian of $\rho_E$  is $4$.
\item[(iii)] If $d(d-6)\geq 4(g-1)$ and $E$ has a section which vanishes at exactly two points,
then $\rho_E$ cannot be written as a sum of squares of polynomials.
\end{itemize}\label{one}
\end{theorema}

\begin{remark}  A large class of curves equipped with an antiholomorphic involution with no fixed points is obtained by considering curves in noncompact complex
surfaces equipped with such an involution, e.g. the total space of $\sO(2k)$ on $\oP^1$, or $\oP^1\times \ol{\oP^1}$ minus the diagonal.
\end{remark}
\begin{remark}
 In (ii) the codimension means the topological codimension (equivalently, the codimension of the real algebraic set $Z(\rho_E)$ in $H^0(C,E)^\tau$).
\par
The second statement can be rephrased as follows: $\rho_E$ divides every $5\times 5$ minor, but not every $4\times 4$ minor, of the Hessian of $\rho_E$.\label{remarkii}
\end{remark}
\begin{remark}
 If $E$ is nonspecial, i.e. $h^1(E)=0$, then the numerical condition $d(d-6)\geq 4(g-1)$ is automatically satisfied if $g\geq 1$ (and if $g=0$, $d\geq 6$). Indeed, using the Riemann-Roch theorem for vector bundles (see, e.g., \cite[Corollary III.12]{Beau}) we can write  $d=2(g-1)+n$, $n=h^0(E)$. It follows   that $d(d-6)\geq 4(g-1)$ if $n\geq 6$. The case $n<6$ cannot occur. Indeed, $n$ must be even (since $d$ is even, owing to $\sigma$ having no fixed points) and there are no very ample rank $2$ vector bundles with $h^0(E)=2,4$ on curves of positive genus (see, e.g., the proof of Lemma \ref{codim=2} below). 
\end{remark}
\begin{remark}  The condition that $E$ has a section vanishing at exactly two points is satisfied, for example, if $E$ is $2$-very ample, i.e. the evaluation map $H^0(C,E)\to H^0(C,E\otimes\sO_D)$ is surjective for any effective divisor $D$ of degree $3$.
\end{remark}

\section{Proof of Theorem \ref{one}}

We shall first prove  statements (i) and (ii). The fact that $\deg \rho_E =\deg E$ follows from Theorem 3.10 in Chapter 3 of \cite{GKZ}.
The following lemma proves the first claim in statement (ii). 
\begin{lemma} Let $E$ be a   rank $2$  very ample $\sigma$-bundle on a real projective curve $C$ with no real points. Then the codimension of $V(E)^\tau$ is two, unless  $E$ is $\sO(1)\oplus \sO(1)$ on $\oP^1$. \label{codim=2}
\end{lemma}
\begin{remark}
 Here, and in what follows, the codimension of a real algebraic set is its topological codimension. 
\end{remark}

\begin{proof}
 Suppose first that $h^0(E)=4$. Then $S=\oP V(E)$ is a geometrically ruled (by the $S_x=\oP\{s\in H^0(E);\, s(x)=0\}$)  surface in $\oP^3$. It follows that $b_2(S)=2$ \cite[Prop. III.21]{Beau}, and consequently, the degree of $R_E$ is $2$. Thus $S$ is a quadric and $E\simeq \sO(1)\oplus \sO(1)$ on $\oP^1$.
\par
We can therefore assume that $n=h^0(E)>4$.  $E$ being very ample is equivalent to $\dim H^0(C,E(-x-y))=n-4$  for any $x,y\in C$.  A $\tau$-invariant section, which vanishes at an $x\in C$, must also vanish at $\sigma(x)$.  Therefore, for any $x\in C$, 
$$ V(E)^\tau_x= \{s\in H^0(E)^\tau;\, s(x)=0\}=\{s\in H^0(E);\, s(x)=s(\sigma(x))=0\}^\tau,$$
 has real codimension $4$ in $H^0(E)^\tau$. 
 Let us write $C/\sigma$ for the unoriented real surface obtained as the quotient of $C$ by $\sigma$. Points of 
$C/\sigma$ can be identified with divisors $x+\sigma(x)$ on $C$. We
consider the incidence variety
$$ Y=\{(x+\sigma(x),s)\in C/\sigma\times V(E)^\tau;\; s(x)=s(\sigma(x))=0\},$$ 
and the two projections $p_1,p_2$ onto $C/\sigma$ and onto $V(E)^\tau$, respectively.
The fibres of $p_1$ are the $V(E)^\tau_x$ and have real dimension $n-4$, which implies that $Y$ has dimension $n-2$. The generic fibre of $p_2$ is finite and so $\dim_\oR V(E)^\tau=\dim_\oR Y=n-2$.
\end{proof}
The following simple lemma implies  that $\rho_E$ must be nonnegative:
\begin{lemma} Let $f$ be a real polynomial in $n$ variables. If $f$ changes sign, then $f^{-1}(0)$ has codimension $1$ somewhere.
\end{lemma}
\begin{proof}
 After a linear change of variables, we can assume that there are $a\in \oR^{n-1}$, $y_1,y_2\in \oR$, such that $f(a,y_1)<0<f(a,y_2)$. Therefore, for any $x$ in a small neighbourhood $U_a$ of $a$, we have $f(x,y_1)<0<f(x,y_2)$, and the mean value theorem applied to $F_x(t)=f(x,t)$ shows that, for any $x\in U_a$, there exists a $y_x\in \oR$ such that $f(x,y_x)=0$.
\end{proof}
The statement about the Hessian of $\rho_E$ follows from the following two facts:
\begin{itemize}
 \item[1.] The projective dual of $\oP V(E)$ is $\oP(E^\ast)$ embedded in $\oP(W^\ast)$, $W=H^0(C,E)$, via the projective embedding defined by the line bundle $\sO(1)$ on $\oP(E^\ast)$. This is shown in \cite{GKZ}, Theorem 3.11 in Chapter 3.
\item[2.] Segre's dimension formula \cite{Segre}.
\end{itemize}
These two facts imply the corresponding statement for the Hessian of $R_E$.
The formulation in Remark \ref{remarkii} makes clear that the statement remains true for $\rho_E$.

\subsection{The variety $V^{[2]}(E)$}
In order to prove statement (iii) of Theorem \ref{one}, we shall consider  the variety $V^{[2]}(E)\subset V(E)$ of sections which have at least $2$ zeros (counting multiplicities).

\begin{lemma} Let $E$ be a very ample vector bundle of rank $2$ on a smooth connected projective curve $C$. The variety $ V^{[2]}(E)$ is irreducible. It has codimension $2$   unless $E$ is $\sO(1)\oplus \sO(1)$ on $\oP^1$. Moreover,  $\oP V^{[2]}(E)$ is the singular locus of $\oP V(E)$.
\end{lemma}
\begin{proof}
As in the proof of Lemma \ref{codim=2}, we can assume that $h^0(E)>4$. This time we consider the incidence variety 
$$Y=\{(D,s)\in C_2\times V^{[2]}(E);\; s(D)=0\},$$
where $C_2=C^2/\oZ_2$. Let  $p_1:Y\to C_2$ and $p_2:Y\to V^{[2]}(E)$ be the two projections. $Y$ is the total space of a vector bundle over an irreducible variety, hence irreducible. Therefore $V^{[2]}(E)$ is irreducible and of codimension $2$ (cf. the proof of Lemma \ref{codim=2}).
\par
To show that $\oP V^{[2]}(E)$ is the singular locus of $\oP V(E)$, consider the incidence variety
$$\tilde Y=\{(x,[s])\in C\times \oP V(E);\; s(x)=0\}.$$
The projection $p_2$ onto $\oP V(E)$ is ramified over $\oP V^{[2]}(E)$, which implies the inclusion $\oP V(E)_{\rm sing}\subset\oP V^{[2]}(E)  $. On the other hand, as mentioned at the end of the last subsection,  $\oP V(E)$ is the  projective dual of $X=\oP(E^\ast)$ embedded in $\oP(W^\ast)$, $W=H^0(C,E)$, via the projective embedding defined by the line bundle $\sO(1)$ on $\oP(E^\ast)$. Thus, the hyperplane $H$ corresponding to $s\in H^0(C,E)\simeq  H^0(X,\sO(1))$ is tangent to $X$ at two points (or has a tangency point of order $2$) if and only if $s\in  V^{[2]}(E)$. Therefore $\oP V^{[2]}(E)\subset \oP V(E)_{\rm sing}$.
\end{proof}

\begin{lemma} Let $E$ be a very ample rank $2$ $\sigma$-bundle on a real smooth connected projective curve $C$ with no real points.  If $\rho_E$ is a sum of squares of polynomials, then $\deg V^{[2]}(E)< (\deg E)^2/4$.
\end{lemma}
\begin{proof} The inequality is clearly true for $E=\sO(1)\oplus \sO(1)$ on $\oP^1$, so we exclude this case. 
 Suppose that $\rho_E=\sum_{i=1}^m p_i^2$ for real polynomials $p_i(x)$ ($p_i$ is not a constant multiple of $p_j$ for $i\neq j$). Then the $E$-resultant is of the same form: $R_E(z)=\sum_{i=1}^m p_i^2(z)$. From the previous lemma we can conclude that the subset $P$ of $H^0(C,E)$ defined by $p_i(z)=0$, $i=1,\dots,m$, is contained in $V^{[2]}(E)$.
On the other hand, the variety $V(E)^\tau$ is given by the equations $p_i(x)=0$, $i=1,\dots,m$ and it has codimension $2$, owing to Lemma \ref{codim=2}.  Therefore $P$ has  codimension at most $2$ and, as $V^{[2]}(E)$ is irreducible, $P=V^{[2]}(E)$. The degree of each $p_i$ is at most $d/2$, and since $V^{[2]}(E)$ has codimension $2$, it follows that its degree is at most $(d/2)^2$. Moreover, the irreducibility of $V(E)$ implies that $m\geq 3$, so the inequality is strict.
\end{proof}

Theorem \ref{one}(iii) will now follow from: 
\begin{proposition}
 Let $C$ be a curve of genus $g$ and $E$ a rank $2$ very ample vector bundle on $C$. Let $d=\deg E$ and suppose that $E$ has a section with exactly $2$ zeros. Then the degree of $V^{[2]}(E)$ is $\frac{d(d-3)}{2}+1-g$.\label{degree}
\end{proposition}

In order to prove this, we need to discuss certain natural vector bundles on symmetric powers of $C$.

\subsection{Tautological vector bundles on symmetric products and $\deg V^{[2]}(E)$}

Let $E$ be a vector bundle of rank $r$ on an algebraic curve  $C$. For any $k\geq 1$ one obtains a vector bundle $E^{[k]}$ of rank $kr$ on the $k$-th symmetric product 
$C_k=C^k/\Sigma_k$. Consider
$$ C\stackrel{q}{\longleftarrow} C\times C_{k-1} \stackrel{p}{\longrightarrow} C_k,$$
where $q$ is the projection and $p$ is the symmetrisation map. Then $E^{[k]}=p_\ast q^\ast E$. This construction goes back to Schwarzenberger \cite{Sch}. Mattuck \cite{Mat}  has shown that $E^{[k]}$ can be also constructed as the $\Sigma_k$-invariant part of $f_\ast (q_1^\ast E\oplus\dots \oplus q_k^\ast E)$, where $q_i:C^k\to C$ is the projection onto the $i$-th factor and $f:C^k\to C_k$ is the symmetrisation. The bundle $E^{[k]}$ was called {\em symmetrization of $E$} by Mattuck, but the term ``{\em tautological bundle}`` seems to have been accepted in more recent literature.
\par
The fibre of $E^{[k]}$ at $D\in C_k$ is $H^0(C,E\otimes \sO_D)$. 
The K\"unneth formula implies easily that $H^0(C_k, E^{[k]})\simeq H^0(C,E)$, and so we have the natural evaluation map
\begin{equation}
 H^0(C,E)\to E^{[k]}, \label{eval}
\end{equation}
which is surjective at $D\in C_k$ if and only if the evaluation morphism 
\begin{equation} H^0(C,E)\to H^0(C,E\otimes\sO_D)
\label{evD}
\end{equation}
is surjective.
In particular, if $E$ is very ample, then $E^{[2]}$ is  globally generated. Moreover, the isomorphism $H^0(C_2, E^{[2]})\simeq H^0(C,E)$ induces an isomorphism $V^{[2]}(E)\simeq V(E^{[2]})$ with the variety of sections of $E^{[2]}$ vanishing somewhere. The assumption on $E$ having a section vanishing at exactly two points is equivalent to generic section in $V(E^{[2]})$ having exactly one zero. In this case, the results of Geertsen \cite[Formula 11]{Geer}  imply that the degree of $V^{[2]}(E)\simeq V(E^{[2]})$ is $c_2(E^{[2]})[C_2]$.
\par 
 The Chern character of the tautological bundle $E^{[k]}$ has been computed by Mattuck \cite{Mat}\footnote{Mattuck only considers line bundles $E$, but, owing to a standard argument, this is sufficient to compute the Chern character for higher rank bundles.}, and we simply apply his results (see also \cite[\S VIII.2]{ACGH}). Let $\delta\in H^2(C_2,\oQ)$ be the Poincar\'e dual of the image of the diagonal $\{(p,p)\}$ in $C_2$ (i.e. of $\{2p\}\subset C_2$). Let $x\in H^2(C_2,\oQ)$ be the Poincar\'e dual of $p+C$, and $1$ the Poincar\'e dual of $[C_2]$. These satisfy the following relations:
$$ x^2=1,\quad x\delta=2,\quad \delta^2=4(1-g).$$
The Chern character of $E^{[2]}$, where $\rank E =r$ and $\deg(E)=d$, is 
$$ {\rm ch}(E^{[2]})=d(1-e^{-x})-r(g-1)+ r\left((g+1)(1+x)-\frac{\delta}{2}\right)e^{-x}.$$
Hence:
\begin{equation}
 c_1(E^{[2]})=dx-r\frac{\delta}{2},\quad c_2(E^{[2]})=\frac{d(d+1-2r)}{2}-\frac{r(r-1)}{2}(g-1).\label{Chern} 
\end{equation}
Thus, for a rank $2$ bundle $E$ of degree $d$, the degree of the variety $ V^{[2]}(E)$ is $\frac{d(d-3)}{2}+1-g$, which proves Proposition \ref{degree} and, consequently, Theorem \ref{one}.

\section{Examples}

We discuss the case of $C=\oP^1$, equipped with the antipodal map $\sigma$. If $z$ is the affine coordinate on $\oP^1$, then $\sigma(z)=-1/\ol{z}$. The involution $\sigma$ induces an anti-linear map $\tau$ on $\sO(k)$, which squares to $1$ if $k$ is even, and to $-1$ if $k$ is odd. For $k\geq 0$, its effect on sections,  i.e. on polynomials of degree $k$, is:
\begin{equation}
 \tau \left(\sum_{i=0}^k c_i z^i\right)= \sum_{i=0}^k (-1)^{i+1}\ol{c_{k-i}} z^i.\label{tau}
\end{equation}
It follows that the rank $2$ $\sigma$-bundles on $\oP^1$ are $\sO(2k)\oplus \sO(2l)$ with diagonal action of $\tau$, or $\sO(2k+1)\oplus \sO(2k+1)$ with $\tau(u,v)=(\tau(v),-\tau(u))$. The polynomial $R_E$ is the classical resultant $R(f,g)$ of two polynomials. We compute the real polynomial $\rho_E$ for $E=\sO(i)\oplus \sO(i)$, $i=1,2,3$:

\subsection{ $E=\sO(1)\oplus\sO(1)$}  In this case, a real section of $E$ is a pair $f=a_0+a_1 z$, $g=\ol{a_1}-\ol{a_0}z$. The resultant $R(f,g)$ is $|a_1|^2+|a_0|^2$, which is the quadratic polynomial $x_1^2+x_2^2+x_3^2+x_4^2$ on $\oR^4$.

\subsection{$E=\sO(2)\oplus\sO(2)$} Real sections of $E$ are pairs $f=a+rz-\bar{a}z^2,g=b+sz-\bar{b}z^2$, $a,b\in \oC$, $r,s\in \oR$, and the classical resultant of two such polynomials is
$$ R(f,g)=|as-br|^2-(a\bar b -\bar a b)^2.$$
After a linear change of coordinates, we can write this as the polynomial 
$$ \rho_E(x_1,x_2,x_3,y_1,y_2,y_3)=(x_1y_2-x_2y_1)^2+(x_1y_3-x_3y_1)^2+(x_2y_3-x_3y_2)^2$$
on $\oR^6$.

\subsection{$E=\sO(3)\oplus \sO(3)$} This is the first case where Theorem \ref{one} implies that $\rho_E$ is not a sum of squares of polynomials.
\par
Sections of $E$ are pairs of polynomials $f=\sum_{i=0}^3 a_iz^i,g=\sum_{i=0}^3 b_iz^i$ of degree $3$ and  $R_E=R(f,g)$ can be computed from the Cayley-Bezout formula \cite[\S 12.1]{GKZ}:
\begin{equation}
 R(f,g)=\det\begin{pmatrix}
             [3,0] & [3,1] & [3,2]\\ [2,0] & [3,0]+[2,1] & [3,1]\\ [1,0] & [2,0] & [3,0]
            \end{pmatrix},
\end{equation}
where $[i,j]=a_ib_j-b_ia_j$. We can identify $H^0(C,E)^\tau$ with pairs 
$$f=a_0+a_1z+a_2z^2+a_3z^3,\quad g= \ol{a_3}-\ol{a_2}z+\ol{a_1}z^2-\ol{a_0}z^3.$$
 It follows that on  $H^0(C,E)^\tau$ we have
$$ [3,0]=|a_3|^2+|a_0|^2,\quad [2,1]=-|a_2|^2-|a_1|^2,$$ $$ [3,2]=\ol{[1,0]}=a_3\ol{a_1}+a_2\ol{a_0},\quad 
 [3,1]=-\ol{[2,0]}=-a_3\ol{a_2}+a_1\ol{a_0}.
$$
Writing 
$$r=[3,0],\quad s=-[2,1],\quad u=[3,1],\quad v=[3,2],$$
 we obtain
\begin{equation}
 \rho_E(f)=R(f,\sigma(f))=r^2(r-s)+2r|u|^2-(r-s)|v|^2 +u^2\bar{v}+\bar{u}^2v.
\end{equation}
It is a polynomial of degree $6$ in $8$ real variables, with $224$ terms.
Using the identity $|u|^2+|v|^2=rs$, we can eliminate $s$ (if $r\neq 0$), and  rewrite $\rho_E$ in the following (nonpolynomial) form, which makes its nonnegativity obvious:
\begin{equation}
  \rho_E(f)= \frac{1}{r}\left((r^2-|v|^2)^2+|r\bar u+u\bar v|^2\right).\label{rho3}
\end{equation} 
This also shows  that the vanishing of $\rho_E$ implies the equality $r\bar u+u\bar v=0$, i.e. 
$$  (|a_0|^2+|a_3|^2)(a_0\ol{a_1}-a_2\ol{a_3})=(a_3\ol{a_2}-a_1\ol{a_0})(a_0\ol{a_2}+a_1\ol{a_3}).$$
The reverse implication is valid provided $u\neq 0$.

\end{document}